\documentclass[10pt,a4paper]{article}
\usepackage{amsmath}
\usepackage[utf8]{inputenc}
\usepackage[T1]{fontenc}
\usepackage{lmodern}
\usepackage{amsmath,amsthm,amssymb}
\usepackage{graphicx}
\usepackage[all,poly,knot]{xy}

\usepackage{listings}
\usepackage{tikz-cd}

\usepackage{booktabs}

\usepackage{color}

\usepackage{microtype}
\usepackage[hidelinks]{hyperref}
\renewcommand{\arraystretch}{1.7}
\usepackage[a4paper,textwidth=162mm,textheight=235mm,centering]{geometry}
\renewcommand{\theequation}{\thesection.\arabic{equation}}

\newtheorem{lem}{Lemma}[section]
\newtheorem{thm}{Theorem} [section]

\theoremstyle{definition}
\newtheorem{exmp}{Example} [section]
\theoremstyle{plain}
\newtheorem{coro}{Corollary}[section]
\theoremstyle{remark}
\newtheorem{rem}{Remark}[section]
\theoremstyle{plain}

\newtheorem{prob}{Problem}[section]

\title{On cyclically covering subspaces of $\mathbb{F}^n_q$}

\author{Yangcheng Li$^{1,}$\footnote{E-mail\,$:$ liyc@m.scnu.edu.cn.} \,\, Pingzhi Yuan$^{1,}$\footnote{Corresponding author. E-mail\,$:$ yuanpz@scnu.edu.cn. Supported by the National Natural Science Foundation of China (grant numbers 12571003 and 12501006) and the Basic and Applied Basic Research Foundation of Guangdong Province (grant number 2024A1515010589).} \,\, Shuang Li$^{1,}$\footnote{E-mail\,$:$ 2338482253@qq.com.}  \,\, Yuanpeng Zeng$^{1,}$\footnote{E-mail\,$:$ zengyp2025@163.com.}   \\
	{\small\it  $^{1}$School of Mathematical Sciences, South China Normal University,}\\
	{\small\it Guangzhou 510631, Guangdong, P. R. China} \\
}

\date{}
\begin{document}
\baselineskip15pt \maketitle
\renewcommand{\theequation}{\arabic{section}.\arabic{equation}}
\catcode`@=11 \@addtoreset{equation}{section} \catcode`@=12

\begin{abstract}
	For a prime power $q$ and a positive integer $n$, let $h_q(n)$ be
	the largest codimension of a subspace of $\mathbb F_q^n$ whose
	cyclic shifts cover the whole space. Using invariant decompositions
	and trace pairings, we prove that $h_q(\ell^t)=0$ whenever $q>2$
	is a primitive root modulo the odd prime power $\ell^t$.
	We also show that $h_q(2n)=0$ if and only if $h_q(n)=0$ when
	$q$ and $n$ are odd. Multiplicative-subgroup constructions yield
	further exact values, and we establish $h_3(11)=h_3(16)=1$.
	Finally, we prove $h_{q^m}(n)\leq h_q(n)$, so every zero value
	is preserved under field extension.
\end{abstract}

{\bf Keywords:} Finite fields, Cyclically covering subspaces, Cyclic shift, Codimension.

\section{Introduction}
A finite field, denoted as $\mathbb{F}_{q}$, where $q$ is a prime power, is a field that contains a finite number of elements. For $n\in\mathbb{N}$, let $\{e_0, e_1, \dots, e_{n-1}\}$ be the standard basis for $\mathbb{F}^n_q$, the indices of vectors in $\mathbb{F}^n_q$ will be taken modulo $n$ (in particular, we set $e_n = e_0$). Define the cyclic shift operator $\tau : \mathbb{F}^n_q\to \mathbb{F}^n_q$ by
\[\tau: \sum_{i=0}^{n-1}a_ie_i \mapsto \sum_{i=0}^{n-1}a_ie_{i+1}.\]
We say that a subspace $U \subset \mathbb{F}^n_q$ is cyclically covering if $\bigcup_{i=0}^{n-1}\tau^i(U) =\mathbb{F}^n_q$. For any $n\in\mathbb{N}$, let $h_q(n)$ denote the largest possible codimension of a cyclically covering subspace of $\mathbb{F}^n_q$.

For general \(n\), determining the value of \(h_q(n)\) is quite difficult. Relatively few values of \(h_q(n)\) have been determined. Motivated by Isbell's conjecture \cite{Isbell56, Isbell60}, Cameron, Ellis, and Raynaud \cite{Cameron-Ellis-Raynaud} in 2019 studied several properties of \(h_q(n)\), determined some of its special values, and established upper and lower bounds for it. Their main results are summarized as follows.
\begin{lem}\cite{Cameron-Ellis-Raynaud} \label{lem1.1} Let \(q\) be a power of prime \(p\), and \(n,m,d,k \in \mathbb{N}\), then the following hold.

$(\mathrm{i})$ $h_2(n) \ge 2$ for all odd integers $n > 3$.

$(\mathrm{ii})$ $h_q(nm) \geq \max\{h_q(n), h_q(m)\}.$

$(\mathrm{iii})$ \(h_q(n) \leq \lfloor\log_q(n)\rfloor\).

$(\mathrm{iv})$ $h_q(q^d - 1) = d - 1 = \lfloor\log_q(q^d - 1)\rfloor.$

$(\mathrm{v})$ \(h_q(M/c) \geq kd + k - c\frac{q^k - 1}{q - 1}\), where \(M = (q - 1)\left(\sum_{r=0}^d q^{kr}\right)\), and \(M\) has a divisor \(c \in \mathbb{N}\) such that \(c < (q - 1)\frac{q^k - q^{-kd}}{q^k - 1}\).

$(\mathrm{vi})$ \(h_q\left(\sum_{r=0}^d q^{kr}\right) = kd\), if \(\gcd(d + 1, q^k - 1) = 1\).

$(\mathrm{vii})$ $h_q(kp^d) = 0$ if \(k \mid q - 1\).

$(\mathrm{viii})$ \(h_2(n) = 0\) if and only if \(n = 2^d\) for some \(d \in \mathbb{N} \cup \{0\}\), and \(h_2(n) = 1\) if and only if \(n = 3\).
\end{lem}

In 1991, Cameron (see \cite{Cameron} Problem 190) posed the following problem in an equivalent form:
\begin{prob}\cite{Cameron}\label{prob1.1}
Does $h_2(n)\to\infty$ as $n\to\infty$ over the odd integers or is $h_2(n) =2$ for infinitely many odd $n$?
\end{prob}
Motivated by Problem \ref{prob1.1}, in 2019, Aaronson, Groenland and Johnston \cite{Aaronson-Groenland-Johnston} investigated the cyclically covering subspaces of \(\mathbb{F}_2^n\). Their main conclusions are summarized as follows.
\begin{lem}\cite{Aaronson-Groenland-Johnston} \label{lem1.2} Let \(q\) be a power of prime \(p\), and \(n,m,\ell \in \mathbb{N}\), then the following hold.

$(\mathrm{i})$ \(h_q(mn) \geq h_q(m) + h_q(n)\).

$(\mathrm{ii})$ \(h_q(pn) \leq ph_q(n)\).

$(\mathrm{iii})$ \(h_q(\ell p^d) = 0\) for any \(\ell < q\).

$(\mathrm{iv})$ \(h_2(t) = 2\) if \(t > 3\) is a prime for which \(2\) is a primitive root.

$(\mathrm{v})$ \(h_q(t) = 0\), where \(q\) is an odd prime, and \(t > q\) is a prime with \(q\) as a primitive root.
\end{lem}
Conditional on Artin's conjecture, Lemma \ref{lem1.2}(iv) gives infinitely many odd primes $p$ with $h_2(p)=2$. Thus $h_2(n)$ does not tend to infinity along the odd integers under that conjecture. Artin conjectured that $2$ is a primitive root modulo infinitely many primes. More generally, for any non-square positive integer \(n\), there are infinitely many primes \(p\) for which \(n\) is a primitive root modulo \(p\). Widely believed, the conjecture follows from the generalized Riemann hypothesis (Hooley \cite{Hooley}). Though no \(n\) is known to satisfy it, Heath-Brown \cite{Heath-Brown} proved it holds for at least one of \(\{2, 3, 5\}\).

Furthermore, Aaronson, Groenland and Johnston \cite{Aaronson-Groenland-Johnston} posed several interesting and challenging problems, such as:
\begin{prob}\cite{Aaronson-Groenland-Johnston}\label{prob1.2}
For which \( n \in \mathbb{N} \) is \( h_q(n) = 0 \)?
\end{prob}

In 2024, Huang \cite{Huang} obtained a necessary and sufficient condition for \(h_q(n) = 0\) when \(\gcd(q, n) = 1\).
His main result shows that this problem can be fully reduced to computing the values of the trace function over finite fields.
He also derived the following conclusions.
\begin{lem}\cite{Huang}\label{lem1.3} Let \( p \) be an odd prime, and let \( q \) be a power of \( p \). For any non-negative integer \( d \), the following statements hold:

$(\mathrm{i})$ \( h_q\left(p^d(q + 1)\right) = 0 \).

$(\mathrm{ii})$ \( h_q\left(2p^d(q - 1)\right) = 0 \) if \(4 \mid q + 1 \).

$(\mathrm{iii})$ Let \( \ell \) be an odd prime number such that \( q \) is a primitive root modulo \( 2\ell \). If \( q \) is relatively prime to \( \ell - 1 \), then \( h_q\left(2p^d\ell\right) = 0 \).
\end{lem}

In 2025, Sun, Ma and Zeng \cite{Sun-Ma-Zeng} investigated the cyclically covering subspaces of the finite field \(\mathbb{F}_{q^n}\) and determined the value of \(h_2(n)\) for certain special values of \(n\). In particular, they showed that \(h_2(21) = 4\). Finally, they established several lower bounds for \(h_q(n)\) in the case where \(\gcd(q, n) = 1\).

In addition, Li and Yuan \cite{Li-Yuan} proved that the problem of determining \(h_q(n) = 0\) can be reduced to the case where \(\gcd(q, n) = 1\). Specifically, they established the following result.
\begin{lem}\cite{Li-Yuan}\label{lem1.4} 
Let \( q \) be a power of a prime $p$, and let \( n \) be a positive integer satisfying \( \gcd(p, n) = 1 \). Then for any non-negative integer \( k \), we have \( h_q(np^k) = 0 \) if and only if \( h_q(n) = 0 \).
\end{lem}

We also mention a problem related to covering vector spaces over \(\mathbb{F}_q\). Luh \cite{Luh} proved that any vector space \(V\) of dimension at least two over a finite field \(\mathbb{F}_q\) can be expressed as the union of \(|\mathbb{F}_q| + 1\) proper subspaces, and such a collection of subspaces is unique up to automorphisms of \(V\).

In this paper, we study zero values of $h_q(n)$ using invariant
decompositions and trace pairings. Theorem~\ref{lem4.1} reduces
the zero-value problem to the corresponding questions on invariant
components. Applying trace criteria to these components, we obtain
vanishing results at prime-power lengths under primitive-root
hypotheses for $q>2$. We also prove that, for odd $q$, doubling an
odd length preserves the zero-value property. Multiplicative-subgroup
constructions yield exact values, including a necessary and sufficient
condition for $h_q(n)=1$ when $n>q$ and $n\mid q^2-1$.
We further establish $h_3(11)=h_3(16)=1$.
Finally, we compare covering invariants over a field and its extensions,
showing in particular that $h_{q^m}(n)\leq h_q(n)$.
Throughout, a \emph{proper} or \emph{nontrivial} covering means a
covering subspace different from the whole ambient space.

\section{Preliminaries}
In this section, we establish several isomorphisms between vector spaces over \(\mathbb{F}_q\). Some of these results can be found in \cite{Sun-Ma-Zeng}; however, for the sake of completeness, we restate them here.

Let us recall the basic notation. For $n\in\mathbb{N}$, let $\{e_0, e_1, \dots, e_{n-1}\}$ be the standard basis for $\mathbb{F}^n_q$, the indices of vectors in $\mathbb{F}^n_q$ will be taken modulo $n$ (in particular, we set $e_n = e_0$). Define the cyclic shift operator $\tau : \mathbb{F}^n_q\to \mathbb{F}^n_q$ by
\[\tau: \sum_{i=0}^{n-1}a_ie_i \mapsto \sum_{i=0}^{n-1}a_ie_{i+1}.\]
We say that a subspace $U \subset \mathbb{F}^n_q$ is cyclically covering if $\bigcup_{i=0}^{n-1}\tau^i(U) =\mathbb{F}^n_q$.

{\bf Isomorphism I}. First, we establish an isomorphism between \(\mathbb{F}_q^n\) and \(\mathbb{F}_{q^n}\). Let \(\{\gamma, \gamma^q, \ldots, \gamma^{q^{n-1}}\}\) be a normal basis of \(\mathbb{F}_{q^n}\) over \(\mathbb{F}_q\). Then, the isomorphism $\varphi: \mathbb{F}_q^n\to \mathbb{F}_{q^n}$ is given by
\[\varphi: \sum_{i=0}^{n-1}a_ie_i \mapsto \sum_{i=0}^{n-1}a_i\gamma^{q^{i}}.\]
Let \(\sigma\) be the Frobenius automorphism of \(\mathbb{F}_{q^n}\), i.e., \(\sigma(\alpha) = \alpha^q\) for all \(\alpha \in \mathbb{F}_{q^n}\). Thus, we have \(\sigma \circ \varphi = \varphi \circ \tau\), i.e., the following commutative diagram commutes:
\[\begin{tikzcd}[row sep=2.2em, column sep=3.5em, arrows=-stealth]
\mathbb{F}_q^n 
\arrow[r, "\tau"] 
\arrow[d, "\varphi"']
& \mathbb{F}_q^n
\arrow[d, "\varphi"'] \\
\mathbb{F}_{q^n}
\arrow[r, "\sigma"]
& \mathbb{F}_{q^n}
\end{tikzcd}\]	
If $U$ is cyclically covering, then
\[\mathbb{F}_{q^n}=\varphi\left(\mathbb{F}_q^n\right)=\varphi\left(\bigcup_{i=0}^{n-1}\tau^i(U)\right)=\bigcup_{i=0}^{n-1}\varphi\left(\tau^i(U)\right)=\bigcup_{i=0}^{n-1}\sigma^i\left(\varphi(U)\right).\]
Therefore, \(U\) is a cyclically covering subspace of \(\mathbb{F}_q^n\) if and only if \(\varphi(U)\) is a cyclically covering subspace of \(\mathbb{F}_{q^n}\). Furthermore, assuming \(\tau(U) \subset U\), we obtain 
\[\sigma\circ\varphi(U)=\varphi\circ\tau(U)\subset\varphi(U).\]
This means that \(U\) is a \(\tau\)-invariant subspace of \(\mathbb{F}_q^n\) if and only if \(\varphi(U)\) is a \(\sigma\)-invariant subspace of \(\mathbb{F}_{q^n}\).

{\bf Isomorphism II}. Let $\mathbb{F}_q[x]$ be the ring of polynomials in a single indeterminate $x$ over $\mathbb{F}_q$. We consider the \(n\)-dimensional vector space \(\mathbb{F}_q[x]/(x^n - 1)\) over \(\mathbb{F}_q\). The isomorphism $\rho: \mathbb{F}_q^n\to \mathbb{F}_q[x]/(x^n - 1)$ is given by
\[\rho: \sum_{i=0}^{n-1}a_ie_i \mapsto \sum_{i=0}^{n-1}a_ix^{i}+(x^n-1).\]
Consider the map \(\eta: \mathbb{F}_q[x]/(x^n - 1) \to \mathbb{F}_q[x]/(x^n - 1)\) defined by
\[\eta : f(x) + (x^n - 1) \mapsto x f(x) + (x^n - 1).\]
Then, we have \(\eta \circ \rho = \rho \circ \tau\), i.e., the following commutative diagram commutes:
\[\begin{tikzcd}[row sep=2.2em, column sep=3.5em, arrows=-stealth]
\mathbb{F}_q^n 
\arrow[r, "\tau"] 
\arrow[d, "\rho"']
& \mathbb{F}_q^n
\arrow[d, "\rho"'] \\
\mathbb{F}_q[x]/(x^n - 1)
\arrow[r, "\eta"]
& \mathbb{F}_q[x]/(x^n - 1)
\end{tikzcd}\]		
If $U$ is cyclically covering, then
\[\mathbb{F}_q[x]/(x^n-1)=\rho\left(\mathbb{F}_q^n\right)=\rho\left(\bigcup_{i=0}^{n-1}\tau^i(U)\right)=\bigcup_{i=0}^{n-1}\rho\left(\tau^i(U)\right)=\bigcup_{i=0}^{n-1}\eta^i\left(\rho(U)\right).\]
Therefore, \(U\) is a cyclically covering subspace of \(\mathbb{F}_q^n\) if and only if \(\rho(U)\) is a cyclically covering subspace of \(\mathbb{F}_q[x]/(x^n - 1)\). Furthermore, assuming \(\tau(U) \subset U\), we obtain 
\[\eta\circ\rho(U)=\rho\circ\tau(U)\subset\rho(U).\]
This means that \(U\) is a \(\tau\)-invariant subspace of \(\mathbb{F}_q^n\) if and only if \(\rho(U)\) is a \(\eta\)-invariant subspace of \(\mathbb{F}_q[x]/(x^n - 1)\).

{\bf Isomorphism III}. Suppose that \(x^n - 1\) factors over \(\mathbb{F}_q\) as
\[x^n - 1 = f_1(x)f_2(x)\cdots f_r(x),\]
where the polynomials \(f_i(x)\) are powers of distinct monic irreducible polynomials and are therefore pairwise coprime. By the Chinese Remainder Theorem, the isomorphism $\delta: \mathbb{F}_q[x]/(x^n - 1)\to \bigoplus_{i=1}^r \mathbb{F}_q[x]/(f_i(x))$ is given by
\[\delta: g(x) + (x^n - 1) \mapsto \left(g(x) + (f_1(x)), g(x) + (f_2(x)), \cdots, g(x) + (f_r(x))\right).\]
Consider the map $\eta_1: \bigoplus\limits_{i=1}^r \mathbb{F}_q[x]/(f_i(x)) \to \bigoplus\limits_{i=1}^r \mathbb{F}_q[x]/(f_i(x))$ defined by
\[\eta_1: \left(g_1(x)+(f_1(x)), \cdots, g_r(x)+(f_r(x))\right) \mapsto \left(xg_1(x)+(f_1(x)), \cdots, xg_r(x)+(f_r(x))\right).\]
Then, we have \(\eta_1 \circ \delta = \delta \circ \eta\), i.e., the following commutative diagram commutes:
\[\begin{tikzcd}[row sep=2.2em, column sep=3.5em, arrows=-stealth]
\mathbb{F}_{q}[x]/(x^n-1) 
\arrow[r, "\eta"] 
\arrow[d, "\delta"']
& \mathbb{F}_{q}[x]/(x^n-1)
\arrow[d, "\delta"'] \\
\bigoplus\limits_{i=1}^r \mathbb{F}_q[x]/(f_i(x))
\arrow[r, "\eta_1"]
& \bigoplus\limits_{i=1}^r \mathbb{F}_q[x]/(f_i(x))
\end{tikzcd}\]
If $V$ is a covering under multiplication by $x$, then
\[\bigoplus\limits_{i=1}^r \mathbb{F}_q[x]/(f_i(x))=\delta\left(\mathbb{F}_{q}[x]/(x^n-1)\right)=\delta\left(\bigcup_{i=0}^{n-1}\eta^i(V)\right)=\bigcup_{i=0}^{n-1}\delta\left(\eta^i(V)\right)
=\bigcup_{i=0}^{n-1}\eta_1^i\left(\delta(V)\right).\]
Therefore, \(V\) is a cyclically covering subspace of \(\mathbb{F}_{q}[x]/(x^n-1)\) if and only if \(\delta(V)\) is a cyclically covering subspace of \(\bigoplus\limits_{i=1}^r \mathbb{F}_q[x]/(f_i(x))\). Furthermore, assuming \(\eta(V) \subset V\), we obtain 
\[\eta_1\circ\delta(V)=\delta\circ\eta(V)\subset\delta(V).\]
This means that \(V\) is a \(\eta\)-invariant subspace of \(\mathbb{F}_{q}[x]/(x^n-1)\) if and only if \(\delta(V)\) is a \(\eta_1\)-invariant subspace of \(\bigoplus\limits_{i=1}^r \mathbb{F}_q[x]/(f_i(x))\).

{\bf Isomorphism IV}. Suppose that \(f(x)\) is a degree \(d\) irreducible polynomial over \(\mathbb{F}_q\), and let \(f(x)\) be the minimal polynomial of \(\omega\) over \(\mathbb{F}_q\). The following commutative diagram commutes:
\[\begin{tikzcd}[row sep=2.2em, column sep=3.5em, arrows=-stealth]
\mathbb{F}_q[x]/(f(x)) 
\arrow[r, "\mu"] 
\arrow[d, "\eta_2"']
& \mathbb{F}_{q}[\omega] \cong \mathbb{F}_{q^{d}} 
\arrow[d, "\eta_3"', shift left=-3ex] \\
\mathbb{F}_q[x]/(f(x)) 
\arrow[r, "\mu"]
& \mathbb{F}_{q}[\omega] \cong \mathbb{F}_{q^{d}}
\end{tikzcd}\]
where
\begin{equation*}
\begin{split}
\mu:&~\mathbb{F}_{q}[x]/(f(x))\to \mathbb{F}_{q}[\omega], \quad \mu(g(x)+(f(x)))=g(\omega);\\
\eta_2:&~\mathbb{F}_q[x]/(f(x))\to\mathbb{F}_q[x]/(f(x)), \quad \eta_2(g(x)+(f(x)))=xg(x)+(f(x));\\
\eta_3:&~\mathbb{F}_{q}[\omega]\to\mathbb{F}_{q}[\omega], \quad \eta_3(g(\omega))=\omega g(\omega).
\end{split}
\end{equation*}

The proofs of our results rely on the isomorphisms between the several vector spaces given above, which allow us to transform the problem among these vector spaces.

\section{\texorpdfstring{Equivalent conditions for $h_q(n)=0$}{Equivalent conditions for zero covering codimension}}
We first separate a general fact about invariant direct sums from the finite-field criteria used below. Let $V$ be a finite-dimensional vector space over $\mathbb F_q$, and let $\sigma\in\operatorname{GL}(V)$ have order $t$. A subspace $U\leq V$ is called \emph{$\sigma$-covering} if
\[
 V=\bigcup_{j=0}^{t-1}\sigma^jU.
\]
We write $h_\sigma(V)$ for the largest codimension of a $\sigma$-covering subspace. Thus $h_\sigma(V)=0$ means that the whole space is the only covering subspace. When $W$ is invariant, $h_\sigma(W)$ always refers to the restricted action on $W$.

Two elementary observations will be used throughout. First, if $U$ is a $\sigma$-covering of $V$ and $W\leq V$ is invariant, then
\[
 W=W\cap\bigcup_{j=0}^{t-1}\sigma^jU
   =\bigcup_{j=0}^{t-1}\sigma^j(U\cap W).
\]
Thus $U\cap W$ covers $W$. This observation does not require an invariant complement to $W$.

Second, suppose $V=W\oplus W'$ is an invariant direct sum and $A\leq W$ covers $W$. Then $A\oplus W'$ covers $V$. Indeed, for $v=w+w'$, choose $j$ and $a\in A$ with $w=\sigma^ja$. Since $W'$ is invariant, $\sigma^{-j}w'\in W'$, and
\[
 v=\sigma^j(a+\sigma^{-j}w')\in\sigma^j(A\oplus W').
\]
Moreover, $\operatorname{codim}_V(A\oplus W')=\operatorname{codim}_W A$. We refer to this construction as lifting a covering from one summand.

\begin{thm}\label{lem4.1}
Suppose that
\[
 V=W_1\oplus\cdots\oplus W_s,
\]
where each $W_i$ is $\sigma$-invariant. Then
\begin{equation}\label{eq:direct-sum-zero}
 h_\sigma(V)=0
 \quad\Longleftrightarrow\quad
 h_\sigma(W_i)=0\quad\text{for every }1\leq i\leq s.
\end{equation}
Equivalently, $V$ admits a proper covering subspace if and only if at least one summand admits a proper covering subspace.
\end{thm}

\begin{proof}
Suppose first that $h_\sigma(W_i)=0$ for every $i$, and let $U$ be any $\sigma$-covering of $V$. The intersection observation shows that $U\cap W_i$ covers $W_i$. Since the whole summand is its only covering subspace, $U\cap W_i=W_i$. Hence every $W_i$ is contained in $U$, and therefore $U=V$. This proves $h_\sigma(V)=0$.

Conversely, suppose that $h_\sigma(W_i)>0$ for some $i$, and choose a proper covering $A_i<W_i$. Lifting from this summand shows that
\[
 A_i\oplus\bigoplus_{j\ne i}W_j
\]
is a covering subspace of $V$. It is proper because $A_i<W_i$, so $h_\sigma(V)>0$. This proves the converse.
\end{proof}

We now return to cyclic coverings and the finite-field models of Section~2. Huang~\cite{Huang} also gives a necessary and sufficient condition for $h_q(n)=0$. Our first criterion describes the multiplication action on a field component. If $f$ is an irreducible divisor of $X^n-1$ and $\omega$ is a root, then Isomorphism~IV identifies multiplication by $X$ on $\mathbb F_q[X]/(f)$ with multiplication by $\omega$ on $\mathbb F_q(\omega)$.

\begin{thm}\label{thm3.2}
Let $E=\mathbb F_q(\omega)$ have degree $d$ over $\mathbb F_q$, suppose $\omega^n=1$, and put $T(x)=\omega x$. Write
\[
 \omega^i=\sum_{j=0}^{d-1}a_{ij}\omega^j,
 \qquad 0\leq i<n.
\]
Then $E$ admits a $T$-covering hyperplane if and only if, for every
$(x_0,\ldots,x_{d-1})\in(\mathbb F_q^\times)^d$, there is an index $i$ with $d\leq i<n$ such that
\begin{equation}\label{eq3.2}
 \sum_{j=0}^{d-1}a_{ij}x_j=0.
\end{equation}
In fact, either every hyperplane is $T$-covering or none is.
\end{thm}

\begin{proof}
The trace pairing is nondegenerate, so every hyperplane has the form
\[
 H_\alpha=\{x\in E:\operatorname{Tr}_{E/\mathbb F_q}(\alpha x)=0\},
 \qquad \alpha\in E^\times.
\]
Since $T^iH_\alpha=H_{\alpha\omega^{-i}}$ and the powers $\omega^{-i}$ run through the same subgroup as the powers $\omega^i$, its images cover $E$ precisely when
\[
 \text{for every }y\in E\text{ there exists }0\leq i<n
 \text{ with }\operatorname{Tr}_{E/\mathbb F_q}(\omega^iy)=0.
\]
Here we have substituted $y=\alpha x$. This condition is independent of $\alpha$.
Let $\beta_0,\ldots,\beta_{d-1}$ be the trace-dual basis to
$1,\omega,\ldots,\omega^{d-1}$, and write $y=\sum_jx_j\beta_j$. Then
\[
 \operatorname{Tr}_{E/\mathbb F_q}(\omega^iy)=\sum_j a_{ij}x_j.
\]
If $x_j=0$ for some $j$, the index $i=j$ already makes the trace zero. If all coordinates are nonzero, none of the first $d$ powers does so. The remaining requirement is exactly~\eqref{eq3.2}.
\end{proof}

Any proper covering subspace is contained in a covering hyperplane. Thus, to prove that a component has covering invariant zero, it suffices to show that no hyperplane covers it. We apply this observation to prime-power lengths, extending the prime-length result of Aaronson, Groenland and Johnston~\cite{Aaronson-Groenland-Johnston}.

\begin{thm}\label{thm4.1}
Let $q>2$ be a prime power, let $\ell$ be an odd prime, and let $t\geq1$. If $q$ is a primitive root modulo $\ell^t$, then
\[
 h_q(\ell^t)=0.
\]
\end{thm}

\begin{proof}
	We first reduce the problem to the irreducible components of the cyclic
	shift. Since $q$ is a primitive root modulo $\ell^t$, it is a primitive
	root modulo $\ell^j$ for every $1\leq j\leq t$. Hence each cyclotomic
	polynomial
	\[
	\Phi_{\ell^j}(X)
	=1+X^{\ell^{j-1}}+\cdots+X^{(\ell-1)\ell^{j-1}}
	\]
	is irreducible over $\mathbb F_q$. The factorization
	\[
	X^{\ell^t}-1=(X-1)\prod_{j=1}^t\Phi_{\ell^j}(X)
	\]
	and the Chinese Remainder Theorem therefore decompose the shift space
	into invariant field components. The component corresponding to $X-1$
	is one-dimensional and has covering invariant zero. By
	Theorem~\ref{lem4.1}, it suffices to prove the same for each component
	corresponding to $\Phi_{\ell^j}$.
	
	Fix $1\leq j\leq t$, and write
	\[
	m=\ell^{j-1},\qquad d=m(\ell-1).
	\]
	Let $\omega$ be a root of $\Phi_{\ell^j}$ in
	$E=\mathbb F_{q^d}$. The action on this component is multiplication
	by $\omega$. By Theorem~\ref{thm3.2}, our task is to find an element
	$y\in E$ such that
	\[
	\operatorname{Tr}_{E/\mathbb F_q}(\omega^a y)\ne0
	\qquad\text{for every }0\leq a<\ell^j.
	\]
	
	To construct $y$, choose
	$c_0,\ldots,c_{\ell-2}\in\mathbb F_q^\times$ such that
	\[
	S=c_0+\cdots+c_{\ell-2}\ne0.
	\]
	This is possible because $q>2$: set
	$c_0=\cdots=c_{\ell-3}=1$, and choose
	$c_{\ell-2}\notin\{0,-(\ell-2)\}$.
	
	Since the minimal polynomial of $\omega$ has degree $d$, the elements
	$1,\omega,\ldots,\omega^{d-1}$ form a basis of $E$ over $\mathbb F_q$.
	Let $\beta_0,\ldots,\beta_{d-1}$ be its trace-dual basis, so that
	\[
	\operatorname{Tr}_{E/\mathbb F_q}(\omega^a\beta_b)
	=\delta_{ab},
	\qquad 0\leq a,b<d.
	\]
	Define
	\[
	y=\sum_{u=0}^{\ell-2}\sum_{r=0}^{m-1}c_u\beta_{um+r}.
	\]
	Then the trace-dual property gives
	\[
	\operatorname{Tr}_{E/\mathbb F_q}(\omega^{um+r}y)=c_u\ne0
	\qquad
	(0\leq u\leq\ell-2,\ 0\leq r<m).
	\]
	This accounts for all exponents $0\leq a<d$.
	
	It remains to consider the exponents $d,\ldots,\ell^j-1$.
	The relation $\Phi_{\ell^j}(\omega)=0$ yields
	\[
	\omega^{d+r}
	=-\sum_{u=0}^{\ell-2}\omega^{um+r},
	\qquad 0\leq r<m.
	\]
	Taking traces against $y$, we obtain
	\[
	\operatorname{Tr}_{E/\mathbb F_q}(\omega^{d+r}y)
	=-\sum_{u=0}^{\ell-2}c_u=-S\ne0.
	\]
	Thus every required trace is nonzero, so this component has no covering
	hyperplane. Since every proper covering subspace is contained in a
	covering hyperplane, its covering invariant is zero.
	
	The same argument applies to every cyclotomic component.
	Theorem~\ref{lem4.1} now gives $h_q(\ell^t)=0$.
\end{proof}

\begin{rem}\label{rem:primitive-lifting}
	If $q$ is a primitive root modulo an odd prime $\ell$, then it remains
	a primitive root modulo every power of $\ell$ if and only if
	\[
	q^{\ell-1}\not\equiv 1 \pmod{\ell^2}.
	\]
	Primitivity modulo $\ell$ alone is therefore insufficient. For example,
	$8$ is a primitive root modulo $3$, but not modulo $9$, and $h_8(9)=1$.
\end{rem}

We next give a criterion directly in the Frobenius model. Let $E=\mathbb F_{q^n}$, let $\sigma(x)=x^q$, and write $\operatorname{Tr}=\operatorname{Tr}_{E/\mathbb F_q}$. For $\alpha\in E^\times$, put
\[
 V_\alpha=\{x\in E:\operatorname{Tr}(\alpha x)=0\}.
\]
The invariance of trace under Frobenius implies
\[
 \sigma^i(V_\alpha)=\{x\in E:\operatorname{Tr}(\sigma^i(\alpha)x)=0\}.
\]
All linear spans, dimensions and dual bases in the following statement are taken over $\mathbb F_q$.

\begin{thm}\label{thm3.1}
Let
\[
 L=\langle\alpha,\sigma(\alpha),\ldots,\sigma^{n-1}(\alpha)\rangle_{\mathbb F_q},
 \qquad t=\dim_{\mathbb F_q}L.
\]
Then $\alpha,\sigma(\alpha),\ldots,\sigma^{t-1}(\alpha)$ is a basis of $L$. Write
\[
 \sigma^i(\alpha)=\sum_{j=0}^{t-1}k_{ij}\sigma^j(\alpha),
 \qquad t\leq i<n,
 \quad k_{ij}\in\mathbb F_q.
\]
The hyperplane $V_\alpha$ is cyclically covering if and only if, for every $(x_0,\ldots,x_{t-1})\in(\mathbb F_q^\times)^t$, some index $i$ with $t\leq i<n$ satisfies
\begin{equation}\label{eq3.1}
 \sum_{j=0}^{t-1}k_{ij}x_j=0.
\end{equation}
In particular, if $t=n$, then $V_\alpha$ is not cyclically covering.
\end{thm}

\begin{proof}
	We first show that
	\[
	\alpha,\sigma(\alpha),\ldots,\sigma^{t-1}(\alpha)
	\]
	form a basis of $L$. Let $s$ be the smallest index such that
	\[
	\sigma^s(\alpha)
	\in\operatorname{span}_{\mathbb F_q}
	\{\alpha,\sigma(\alpha),\ldots,\sigma^{s-1}(\alpha)\}.
	\]
	By minimality, the first $s$ vectors are linearly independent.
	Applying $\sigma$ repeatedly to this dependence shows that every
	subsequent orbit vector lies in their span. Hence their span is $L$,
	and therefore $s=\dim_{\mathbb F_q}L=t$.
	
	Next, we explain why the first $t$ trace values can be chosen
	independently. The nondegenerate trace pairing identifies $E$ with
	its $\mathbb F_q$-linear dual. Moreover, every linear functional on
	$L$ extends to one on $E$. Thus the map
	\[
	E\longrightarrow L^*,\qquad
	x\longmapsto\bigl(z\mapsto\operatorname{Tr}(zx)\bigr)
	\]
	is surjective. Since the first $t$ orbit vectors form a basis of $L$,
	for every tuple $(x_0,\ldots,x_{t-1})\in\mathbb F_q^t$ there exists
	$x\in E$ satisfying
	\[
	\operatorname{Tr}(\sigma^j(\alpha)x)=x_j,
	\qquad 0\leq j<t.
	\]
	
	We now translate the covering condition into these trace coordinates.
	The invariance of the trace under $\sigma$ gives
	\[
	x\in\sigma^i(V_\alpha)
	\quad\Longleftrightarrow\quad
	\operatorname{Tr}(\sigma^i(\alpha)x)=0.
	\]
	Consequently, the images of $V_\alpha$ cover $E$ if and only if,
	for every $x\in E$, at least one of these traces vanishes.
	
	If one of $x_0,\ldots,x_{t-1}$ is zero, then $x$ already lies in one
	of the first $t$ images. It therefore remains to consider tuples in
	$(\mathbb F_q^\times)^t$. For each $t\leq i<n$, applying the trace
	pairing to the expansion of $\sigma^i(\alpha)$ in the above basis
	expresses the corresponding trace as the linear combination
	in~\eqref{eq3.1}. Since every tuple can occur, the covering condition
	is exactly the stated condition on these linear combinations.
	
	Finally, if $t=n$, choose $x\in E$ for which
	\[
	\operatorname{Tr}(\sigma^j(\alpha)x)=1,
	\qquad 0\leq j<n.
	\]
	Then $x$ lies in none of the images of $V_\alpha$, so $V_\alpha$
	is not a covering subspace.
\end{proof}

\begin{coro}\label{coro3.1}
With the notation of Theorem~\ref{thm3.1}, suppose $\dim_{\mathbb F_q}L\geq2$ and
\[
 L\setminus\{0\}\subseteq
 \mathbb F_q^\times\cdot
 \{\alpha,\sigma(\alpha),\ldots,\sigma^{n-1}(\alpha)\}.
\]
Then $V_\alpha$ is cyclically covering.
\end{coro}

\begin{proof}
	Fix an arbitrary vector $x\in E$. Consider the $\mathbb F_q$-linear map
	\[
	f_x:L\longrightarrow\mathbb F_q,\qquad
	f_x(z)=\operatorname{Tr}(zx).
	\]
	Its image has dimension at most one. Since $\dim_{\mathbb F_q}L\geq2$,
	the rank--nullity theorem gives
	\[
	\dim_{\mathbb F_q}\ker f_x\geq\dim_{\mathbb F_q}L-1\geq1.
	\]
	We may therefore choose a nonzero vector $z\in L$ such that
	$\operatorname{Tr}(zx)=0$.
	
	By the hypothesis, this vector can be written as
	\[
	z=c\,\sigma^i(\alpha)
	\]
	for some $c\in\mathbb F_q^\times$ and $0\leq i<n$. Hence
	\[
	0=\operatorname{Tr}(zx)
	=c\,\operatorname{Tr}(\sigma^i(\alpha)x).
	\]
	Since $c\ne0$, we obtain
	$\operatorname{Tr}(\sigma^i(\alpha)x)=0$, which means that
	$x\in\sigma^i(V_\alpha)$.
	
	As $x$ was arbitrary, the images of $V_\alpha$ cover $E$.
\end{proof}

The preceding trace criteria can be complemented by constructions inside a
single irreducible component.  The next result gives a condition under which
a one-dimensional subspace covers such a component.

\begin{thm}\label{thm1}
Let $q$ be a prime power, and let $t,n,N$ be positive integers with
$nN=q^t-1$.  If
\[
 \gcd\!\left(N,\frac{q^t-1}{q-1}\right)=1,
\]
then
\[
 h_q(n)=t-1.
\]
In particular, this condition gives a proper cyclically covering subspace
when $t\ge2$.
\end{thm}

\begin{proof}
	We construct a cyclically covering subspace of codimension $t-1$
	and then show that no larger codimension is possible.
	
	Let $E=\mathbb F_{q^t}$, choose a generator $g$ of $E^\times$, and put
	\[
	\omega=g^N,\qquad
	M=\frac{q^t-1}{q-1}.
	\]
	Since $nN=q^t-1$, the element $\omega$ has order $n$. Moreover,
	\[
	H=\langle\omega\rangle=\langle g^N\rangle,
	\qquad
	\mathbb F_q^\times=\langle g^M\rangle.
	\]
	The assumption $\gcd(N,M)=1$ implies that these two subgroups generate
	$E^\times$. As $E^\times$ is abelian, this means
	\[
	E^\times=H\mathbb F_q^\times.
	\]
	Thus every nonzero element of $E$ has the form $\omega^j c$ with
	$0\leq j<n$ and $c\in\mathbb F_q^\times$. Including zero, we obtain
	\[
	E=\bigcup_{j=0}^{n-1}\omega^j\mathbb F_q.
	\]
	In other words, the line $\mathbb F_q\leq E$ is covering under
	multiplication by $\omega$.
	
	We next realize this field action as a component of the cyclic shift
	on $\mathbb F_q^n$. The displayed union shows that every element of
	$E$ belongs to $\mathbb F_q(\omega)$, so $\mathbb F_q(\omega)=E$.
	Consequently, the minimal polynomial $f$ of $\omega$ over $\mathbb F_q$
	has degree $t$. Since $\omega^n=1$, we also have $f\mid X^n-1$.
	
	Now $n\mid q^t-1$ implies that $n$ is relatively prime to the
	characteristic of $\mathbb F_q$, so $X^n-1$ is square-free.
	The Chinese Remainder Theorem therefore gives an invariant decomposition
	\[
	\mathbb F_q^n\cong E\oplus W,
	\]
	where the cyclic shift acts on $E$ as multiplication by $\omega$,
	and $W$ is the sum of the remaining components. Under this identification,
	take
	\[
	U=\mathbb F_q\oplus W.
	\]
	The shifts of $U$ cover the whole space: their projections onto $E$
	are the lines $\omega^j\mathbb F_q$, which cover $E$, while the
	$W$-component remains unrestricted. Moreover,
	\[
	\operatorname{codim}_{\mathbb F_q^n}U
	=\dim_{\mathbb F_q}E-\dim_{\mathbb F_q}\mathbb F_q
	=t-1.
	\]
	Hence $h_q(n)\geq t-1$.
	
	For the reverse inequality, let a cyclically covering subspace have
	codimension $r$. Counting the vectors in its $n$ shifts gives
	\[
	q^n\leq n q^{n-r},
	\qquad\text{so}\qquad
	q^r\leq n.
	\]
	Since $n\leq q^t-1<q^t$, we must have $r\leq t-1$.
	Therefore $h_q(n)=t-1$.
\end{proof}

In multiplicative-group arguments such as this one, exponents are taken
modulo $q^t-1$.  The corresponding product of subgroups covers
$\mathbb F_{q^t}^\times$; zero is added separately.  The same distinction
will be used in the following subfield construction.

\begin{thm}\label{thm4.3}
Let $n\ge2$ satisfy $\gcd(q,n)=1$, put
$m=\operatorname{ord}_n(q)$, and write $nN=q^m-1$.
If $N>1$ and $\gcd(n,N)=1$, then
\[
 h_q(n)\ge m-\operatorname{ord}_N(q).
\]
If $N=1$, then $h_q(n)=m-1$.

Moreover, if $nN=q^2-1$ and $n>q$, then
\[
 \begin{aligned}
 h_q(n)=1
 &\quad\Longleftrightarrow\quad
 \operatorname{lcm}(n,q-1)=q^2-1\\
 &\quad\Longleftrightarrow\quad
 \gcd(N,q+1)=1.
 \end{aligned}
\]
\end{thm}

\begin{proof}
	We first prove the lower bound. If $N=1$, then $n=q^m-1$, and
	Theorem~\ref{thm1} gives $h_q(n)=m-1$. We may therefore assume that
	$N>1$ and $\gcd(n,N)=1$.
	
	Set
	\[
	s=\operatorname{ord}_N(q).
	\]
	Since $N\mid q^m-1$, we have $s\mid m$. Thus
	$E=\mathbb F_{q^m}$ contains the subfield $K=\mathbb F_{q^s}$.
	We will show that $K$, viewed as an $s$-dimensional $\mathbb F_q$-subspace
	of $E$, is covering under multiplication by an element of order $n$.
	
	Write $q^s-1=Nd$. Since $s\mid m$,
	\[
	q^s-1\mid q^m-1=nN,
	\]
	so $d\mid n$. Consequently,
	\[
	\frac{q^m-1}{q^s-1}=\frac nd,
	\qquad
	\gcd\!\left(N,\frac nd\right)=1.
	\]
	Choose a generator $g$ of $E^\times$ and put $\omega=g^N$.
	Then $\omega$ has order $n$, and
	\[
	\langle\omega\rangle=\langle g^N\rangle,
	\qquad
	K^\times
	=\left\langle g^{(q^m-1)/(q^s-1)}\right\rangle
	=\langle g^{n/d}\rangle.
	\]
	The coprimality above implies that these subgroups generate
	$E^\times$. Since the group is abelian, their product is $E^\times$.
	Including zero, we obtain
	\[
	E=\bigcup_{j=0}^{n-1}\omega^jK.
	\]
	Thus $K$ is a covering subspace of $E$ of codimension $m-s$.
	
	Because $\omega$ has order $n$, its minimal polynomial over
	$\mathbb F_q$ has degree $\operatorname{ord}_n(q)=m$ and divides
	$X^n-1$. Hence the Chinese remainder decomposition of
	$\mathbb F_q[X]/(X^n-1)$ contains an invariant component isomorphic to
	$E$, on which multiplication by $X$ corresponds to multiplication
	by $\omega$. Taking $K$ in this component and all other components
	in full gives a cyclically covering subspace of codimension $m-s$.
	Therefore
	\[
	h_q(n)\geq m-s.
	\]
	
	We now prove the quadratic criterion. Suppose that
	$nN=q^2-1$ and $n>q$. Since $q<n<q^2$, the counting bound gives
	$h_q(n)\leq1$. Moreover, all roots of $X^n-1$ lie in
	$\mathbb F_{q^2}$, so its irreducible factors over $\mathbb F_q$
	have degree one or two.
	
	A one-dimensional component has no proper covering subspace.
	On a two-dimensional component $E=\mathbb F_{q^2}$, the action is
	multiplication by an element $\beta$ of some order $r\mid n$.
	Every nonzero proper subspace is a line $c\mathbb F_q$, where
	$c\in E^\times$, and the nonzero vectors covered by its images form
	the set
	\[
	c\langle\beta\rangle\mathbb F_q^\times.
	\]
	The subgroup $\langle\beta\rangle\mathbb F_q^\times$ has order
	$\operatorname{lcm}(r,q-1)$. Hence this line covers $E$ if and only if
	\[
	\operatorname{lcm}(r,q-1)=q^2-1.
	\]
	
	By Theorem~\ref{lem4.1}, the whole cyclic space admits a proper covering
	if and only if one of its components does. Therefore, if $h_q(n)=1$,
	some two-dimensional component satisfies the displayed equality.
	Since $r\mid n$ and both $n$ and $q-1$ divide $q^2-1$, this forces
	\[
	\operatorname{lcm}(n,q-1)=q^2-1.
	\]
	
	Conversely, suppose that this equality holds. Choose an element
	$\omega\in\mathbb F_{q^2}^\times$ of order $n$.
	Since $n>q$, it does not belong to $\mathbb F_q$, so its minimal
	polynomial has degree two. The corresponding component has a covering
	line by the preceding criterion. Lifting this line to the whole cyclic
	space gives a covering subspace of codimension one. Thus $h_q(n)=1$.
	
	Finally, using $n=(q^2-1)/N$ and $q-1=(q^2-1)/(q+1)$, we obtain
	\[
	\operatorname{lcm}(n,q-1)
	=
	\operatorname{lcm}\!\left(
	\frac{q^2-1}{N},\frac{q^2-1}{q+1}
	\right)
	=
	\frac{q^2-1}{\gcd(N,q+1)}.
	\]
	Consequently,
	\[
	h_q(n)=1
	\quad\Longleftrightarrow\quad
	\operatorname{lcm}(n,q-1)=q^2-1
	\quad\Longleftrightarrow\quad
	\gcd(N,q+1)=1.
	\]
\end{proof}

The first bound is the subfield construction from
\cite{Sun-Ma-Zeng}.  The exact criterion in the quadratic case does
not require $\gcd(n,N)=1$.  For example,
\[
 q=19,\qquad n=120,\qquad N=3
\]
gives $nN=q^2-1$, $\gcd(N,q+1)=1$, and hence
$h_{19}(120)=1$, although $\gcd(n,N)=3$.

The next result allows us to compare irreducible components with roots
of the same multiplicative order.  No field isomorphism sending one
specified root to the other is needed.

\begin{thm}\label{1}
Let $f_1,f_2\in\mathbb F_q[X]$ be irreducible polynomials of the same
degree $d$, with nonzero roots $\omega_1,\omega_2$ of the same
multiplicative order.  Let $\eta_i$ denote multiplication by $X$ on
$R_i=\mathbb F_q[X]/(f_i)$.  Then
\[
 h_{\eta_1}(R_1)=h_{\eta_2}(R_2).
\]
In particular, one component admits a proper covering subspace if and
only if the other does.
\end{thm}

\begin{proof}
Identify each $R_i$ with $E=\mathbb F_{q^d}$ by evaluation at a root
of $f_i$ in $E$.  Under these identifications, $\eta_i$ becomes
multiplication by $\omega_i$.  A finite cyclic group has a unique
subgroup of each possible order, so
\[
 \langle\omega_1\rangle=\langle\omega_2\rangle.
\]
For every $\mathbb F_q$-subspace $L\le E$, it follows that
\[
 \bigcup_j\omega_1^j L=\bigcup_j\omega_2^j L.
\]
The covering subspaces, and therefore their maximal codimensions, are
the same for the two multiplication actions.
\end{proof}

For odd characteristic, doubling an odd length preserves the property
that the whole space is the only cyclically covering subspace.

\begin{thm}\label{2}
Let $q$ be an odd prime power and let $n$ be an odd positive integer.
Then
\[
 h_q(2n)=0\quad\Longleftrightarrow\quad h_q(n)=0.
\]
\end{thm}

\begin{proof}
	Let $T$ and $\widetilde T$ denote the cyclic shifts on
	$\mathbb F_q^n$ and $\mathbb F_q^{2n}$, respectively. We decompose
	the latter space into
	\[
	W_+=\{(x,x):x\in\mathbb F_q^n\},
	\qquad
	W_-=\{(x,-x):x\in\mathbb F_q^n\}.
	\]
	Since $q$ is odd, every vector $(a,b)$ has a unique expression
	\[
	(a,b)=(u,u)+(v,-v),
	\qquad
	u=\frac{a+b}{2},\quad v=\frac{a-b}{2}.
	\]
	Thus $\mathbb F_q^{2n}=W_+\oplus W_-$.
	
	Both summands are invariant under $\widetilde T$. Indeed,
	\[
	\widetilde T(x,x)=(Tx,Tx),
	\qquad
	\widetilde T(x,-x)=(Sx,-Sx),
	\]
	where
	\[
	S(x_0,\ldots,x_{n-1})
	=(-x_{n-1},x_0,\ldots,x_{n-2}).
	\]
	Hence the actions on $W_+$ and $W_-$ are identified with $T$ and $S$,
	respectively.
	
	We next show that $T$ and $S$ have the same covering invariant.
	Let
	\[
	D=\operatorname{diag}\bigl((-1)^i\bigr)_{i=0}^{n-1}.
	\]
	Since $n$ is odd, a direct calculation gives
	\[
	S=-DTD^{-1}.
	\]
	Consequently, for every subspace $L\leq\mathbb F_q^n$ and every
	$j\geq0$,
	\[
	S^j(L)
	=(-1)^jDT^jD^{-1}(L)
	=DT^j\bigl(D^{-1}(L)\bigr),
	\]
	because multiplication by a nonzero scalar leaves a linear subspace
	unchanged. Thus $L$ is $S$-covering if and only if $D^{-1}(L)$ is
	$T$-covering. Here the subspace images repeat after $n$ steps, since
	$T^n=I$ and $S^n=-I$. As $D$ preserves codimension, we obtain
	\[
	h_S(\mathbb F_q^n)=h_T(\mathbb F_q^n)=h_q(n).
	\]
	
	Both invariant summands therefore have covering invariant $h_q(n)$.
	Applying Theorem~\ref{lem4.1} to $W_+\oplus W_-$ yields
	\[
	h_q(2n)=0
	\quad\Longleftrightarrow\quad
	h_q(n)=0.
	\]
	No assumption that $\gcd(q,n)=1$ is needed.
\end{proof}

\begin{coro}\label{thm4.2}
Let $q>2$ be a prime power, let $\ell$ be an odd prime, and let
$t\ge1$.  If $q$ is a primitive root modulo $\ell^t$, then
\[
 h_q(2\ell^t)=0.
\]
\end{coro}

\begin{proof}
Theorem~\ref{thm4.1} gives $h_q(\ell^t)=0$.  If $q$ is odd, apply
Theorem~\ref{2}.  If $q$ has characteristic two, apply
Lemma~\ref{lem1.4} with $p=2$ and the odd length $\ell^t$.
\end{proof}

\begin{exmp}\label{e1}
	Table~\ref{tab:ternary-values} lists $h_3(n)$ for $4\leq n\leq19$.
	The results in
	\cite{Cameron-Ellis-Raynaud,Aaronson-Groenland-Johnston,Huang},
	together with Theorems~\ref{thm4.1} and~\ref{2}, give all entries
	except $n=11,16$, which we prove below.
	
\begin{table}[htbp]
	\centering
	\small
	\renewcommand{\arraystretch}{1.15}
	\caption{Values of $h_3(n)$ for $4\leq n\leq19$.}
	\label{tab:ternary-values}
	\begin{tabular}{ccl@{\qquad}ccl}
		\toprule
		$n$ & $h_3(n)$ & Source
		& $n$ & $h_3(n)$ & Source\\
		\midrule
		4 & 0 & \cite{Cameron-Ellis-Raynaud} or \cite{Huang}
		& 12 & 0 & \cite{Aaronson-Groenland-Johnston} or \cite{Huang}\\
		
		5 & 0 & \cite{Aaronson-Groenland-Johnston} or Theorem~\ref{thm4.1}
		& 13 & 2 & \cite{Cameron-Ellis-Raynaud}\\
		
		6 & 0 & \cite{Cameron-Ellis-Raynaud} or \cite{Aaronson-Groenland-Johnston}
		& 14 & 0 & \cite{Huang} or Theorem~\ref{2}\\
		
		7 & 0 & \cite{Aaronson-Groenland-Johnston} or Theorem~\ref{thm4.1}
		& 15 & 0 & \cite{Aaronson-Groenland-Johnston}\\
		
		8 & 1 & \cite{Cameron-Ellis-Raynaud}
		& 16 & 1 & Example~\ref{e1}\\
		
		9 & 0 & \cite{Cameron-Ellis-Raynaud} or \cite{Aaronson-Groenland-Johnston}
		& 17 & 0 & \cite{Aaronson-Groenland-Johnston} or Theorem~\ref{thm4.1}\\
		
		10 & 0 & \cite{Huang} or Theorem~\ref{2}
		& 18 & 0 & \cite{Cameron-Ellis-Raynaud} or \cite{Aaronson-Groenland-Johnston}\\
		
		11 & 1 & Example~\ref{e1}
		& 19 & 0 & \cite{Aaronson-Groenland-Johnston} or Theorem~\ref{thm4.1}\\
		\bottomrule
	\end{tabular}
\end{table}
	
	\medskip
	\noindent\emph{The length $11$.}
	Let $T$ be the right cyclic shift on $V=\mathbb F_3^{11}$.
	We use the cyclic-code criterion of Li and Yuan~\cite[Section~2]{LYCC}.
	For $0\ne a\in V$, write
	\[
	H_a=a^\perp,\qquad
	D_a=\langle a,Ta,\ldots,T^{10}a\rangle_{\mathbb F_3}.
	\]
	Then
	\begin{equation}\label{eq:example-normal-test}
		H_a\text{ is covering}
		\quad\Longleftrightarrow\quad
		D_a\text{ contains no full-weight codeword},
	\end{equation}
	where full weight means weight $11$.
	
	Over $\mathbb F_3$, we have
	\begin{equation}\label{e3}
		\begin{split}
			X^{11}-1&=(X-1)f_1(X)f_2(X),\\
			f_1(X)&=X^5+X^4-X^3+X^2-1,\\
			f_2(X)&=X^5-X^3+X^2-X-1.
		\end{split}
	\end{equation}
	Since $\operatorname{ord}_{11}(3)=5$, both degree-five factors
	are irreducible. Thus
	\begin{equation}\label{e4}
		V=W_0\oplus W_1\oplus W_2,\qquad
		W_0=\ker(T-I),\quad W_i=\ker f_i(T),
	\end{equation}
	with $\dim W_0=1$ and $\dim W_1=\dim W_2=5$.
	The code $D_a$ is precisely the sum of those irreducible
	constituents on which $a$ has a nonzero projection.
	Moreover, $W_0$ is spanned by the all-one vector, and
	$W_1\oplus W_2=W_0^\perp$ is the coordinate-sum-zero subspace.
	
	We first show that each $W_i$ is self-orthogonal.
	The factors satisfy
	\[
	X^5f_i(X^{-1})=-f_{3-i}(X).
	\]
	Since $f_1,f_2$ are coprime, $f_i(T^{-1})$ is invertible on $W_i$.
	For $x,y\in W_i$, invariance of the standard inner product under $T$
	gives
	\[
	0=(f_i(T)x,y)=(x,f_i(T^{-1})y).
	\]
	It follows that $W_i\subseteq W_i^\perp$. Consequently,
	\[
	0=(a,a)=\sum_{k=0}^{10}a_k^2
	=\operatorname{wt}(a)\quad\text{in }\mathbb F_3
	\]
	for every $a\in W_i$.
	Thus every word in $W_i$ has weight divisible by $3$ and cannot
	have full weight. By~\eqref{eq:example-normal-test}, every nonzero
	vector in either $W_i$ defines a covering hyperplane.
	In particular, $h_3(11)\geq1$.
	
	A vector with a nonzero $W_0$ projection generates a cyclic code
	containing the all-one vector. A vector with nonzero projections
	on both $W_1$ and $W_2$ generates a cyclic code containing
	$W_1\oplus W_2$, which contains the full-weight word
	$(2,1,\ldots,1)$. Therefore the nonzero covering hyperplane
	normals are exactly
	\[
	(W_1\cup W_2)\setminus\{0\}.
	\]
	Also, since $W_0$ is orthogonal to $W_i$, self-orthogonality
	and dimensions give
	\[
	W_i^\perp=W_0\oplus W_i.
	\]
	
	Suppose that a codimension-two covering $U$ exists, and put
	$S=U^\perp$. Every nonzero $a\in S$ defines a covering hyperplane,
	since $U\subseteq H_a$. Hence $S\subseteq W_1\cup W_2$.
	A vector space cannot be the union of two proper subspaces,
	so $S\subseteq W_i$ for some $i$. Therefore
	\[
	W_0\oplus W_i=W_i^\perp\subseteq U.
	\]
	Set $Q=W_{3-i}$ and $A=U\cap Q$. Since $\dim U=9$, we obtain
	$\dim A=3$. As $Q$ is invariant, intersecting the covering union
	with $Q$ gives
	\[
	Q=\bigcup_{j=0}^{10}T^jA.
	\]
	We show that such a three-dimensional covering cannot exist.
	
	For a vector space $B$, let $\mathbb P(B)$ denote the set of its
	one-dimensional subspaces. The irreducible action of $T$ on $Q$
	has order $11$. Every orbit on $\mathbb P(Q)$ therefore has
	size $11$: an orbit of size one would give a $T$-invariant line.
	Since
	\[
	|\mathbb P(Q)|=\frac{3^5-1}{3-1}=121,
	\]
	there are eleven orbits $\mathcal O_1,\ldots,\mathcal O_{11}$.
	Put
	\[
	m_i=|\mathbb P(A)\cap\mathcal O_i|.
	\]
	The covering condition forces $A$ to meet every orbit, so
	\[
	m_i\geq1,\qquad
	\sum_{i=1}^{11}m_i=|\mathbb P(A)|=13.
	\]
	After assigning one point to each orbit, only two points remain.
	They either lie in the same orbit or in two different orbits.
	Hence
	\[
	\sum_{i=1}^{11}m_i(m_i-1)\leq6.
	\]
	
	On the other hand, for every $1\leq j\leq10$,
	\[
	\dim(A\cap T^jA)\geq3+3-5=1.
	\]
	Each of these intersections therefore contains a projective point.
	Counting ordered pairs of distinct points of $\mathbb P(A)$
	that belong to the same $T$-orbit gives
	\[
	10
	\leq
	\sum_{j=1}^{10}
	|\mathbb P(A)\cap T^j\mathbb P(A)|
	=
	\sum_{i=1}^{11}m_i(m_i-1)
	\leq6,
	\]
	a contradiction. Here the equality holds because, for two distinct
	points in the same orbit, there is exactly one nonzero exponent
	$j$ modulo $11$ taking the second point to the first.
	
	Thus no codimension-two covering exists. Together with
	$1\leq h_3(11)\leq2$, this proves
	\[
	h_3(11)=1.
	\]
	
	\medskip
	\noindent\emph{The length $16$.}
	Let $T$ now denote the right cyclic shift on $V=\mathbb F_3^{16}$.
	Set
	\[
	V_\pm=\ker(T^8\mp I).
	\]
	Explicitly,
	\[
	V_+=\{(x,x):x\in\mathbb F_3^8\},\qquad
	V_-=\{(x,-x):x\in\mathbb F_3^8\}.
	\]
	Since the characteristic is not two,
	\[
	V=V_+\oplus V_-,
	\qquad \dim V_+=\dim V_-=8.
	\]
	Both summands are $T$-invariant.
	
	Suppose that $U$ is a codimension-two covering, so $\dim U=14$.
	The subspaces
	\[
	U_\pm=U\cap V_\pm
	\]
	cover $V_\pm$ under the restricted action.
	Since $T^8=\pm I$ on $V_\pm$, each $U_\pm$ has at most eight
	distinct images. Counting gives
	\[
	3^8\leq8\cdot3^{\dim U_\pm},
	\]
	and therefore $\dim U_\pm\geq7$. Thus
	\[
	14=\dim U
	\geq\dim U_++\dim U_-
	\geq14,
	\]
	which forces
	\[
	U=U_+\oplus U_-.
	\]
	But then $T^8U=U$, so $U$ itself has at most eight distinct
	cyclic images. Consequently,
	\[
	\left|\bigcup_{j=0}^{15}T^jU\right|
	\leq8\cdot3^{14}<3^{16},
	\]
	contrary to the covering assumption.
	
	Thus no codimension-two covering exists. Finally, the divisibility
	and counting bounds give
	\[
	1=h_3(8)\leq h_3(16)\leq2,
	\]
	and hence
	\[
	h_3(16)=1.
	\]
\end{exmp}

\section{\texorpdfstring{Coverings over $\mathbb F_{q^m}^n$ and $\mathbb F_q^{mn}$}{Coverings over extension and base fields}}
Fix an $\mathbb F_q$-basis of $\mathbb F_{q^m}$. Expanding each coordinate in this basis gives an $\mathbb F_q$-linear isomorphism
\[
 \Phi:\mathbb F_{q^m}^n\longrightarrow\mathbb F_q^{mn}.
\]
Let $\tau_1$ be the cyclic shift of the $n$ extension-field coordinates, let $\tau_2$ be the ordinary cyclic shift of the $mn$ base-field coordinates, and put $\tau_3=\tau_2^m$. The coordinates in the target are ordered in consecutive blocks of size $m$, so
\[
 \Phi\tau_1=\tau_3\Phi.
\]
Thus $\Phi$ takes a cyclic covering over the extension field to a covering under the block shift $\tau_3$. An extension-field subspace of codimension $k$ has base-field codimension $mk$ after applying $\Phi$.

\begin{thm}\label{th4.1}
For all prime powers $q$ and positive integers $m,n$,
\[
 m\,h_{q^m}(n)\leq h_{\tau_3}(\mathbb F_q^{mn}).
\]
In particular, if $h_{\tau_3}(\mathbb F_q^{mn})\leq m-1$, then $h_{q^m}(n)=0$.
\end{thm}
\begin{proof}
Choose a cyclically covering subspace $U\leq\mathbb F_{q^m}^n$ of extension-field codimension $k=h_{q^m}(n)$. The relation $\Phi\tau_1=\tau_3\Phi$ shows that $\Phi(U)$ covers $\mathbb F_q^{mn}$ under $\tau_3$. Its base-field dimension is $m(n-k)$, so its codimension is $mk$. This proves the inequality and its stated consequence.
\end{proof}

The same counting argument as for the ordinary cyclic shift gives a bound that depends on the number of blocks, rather than on the base-field dimension.
\begin{thm}\label{th4.2}
For all positive integers $m,n$,
\[
 h_{\tau_3}(\mathbb F_q^{mn})\leq\lfloor\log_q n\rfloor.
\]
\end{thm}

\begin{proof}
	Let $U$ be a $\tau_3$-covering subspace, and put
	$r=\operatorname{codim}_{\mathbb F_q}U$.
	Since $\tau_3^n=\mathrm{id}$, its first $n$ images cover
	$\mathbb F_q^{mn}$. Each image has the same cardinality as $U$,
	namely $q^{mn-r}$. Therefore
	\[
	q^{mn}
	=\left|\bigcup_{j=0}^{n-1}\tau_3^jU\right|
	\leq\sum_{j=0}^{n-1}|\tau_3^jU|
	=nq^{mn-r}.
	\]
	Cancelling $q^{mn-r}$ gives $q^r\leq n$. As $r$ is an integer,
	\[
	r\leq\lfloor\log_q n\rfloor.
	\]
	Taking the maximum over all $\tau_3$-covering subspaces proves
	the assertion.
\end{proof}

There are two useful comparisons with coverings over the base field. One comes from $\tau_3=\tau_2^m$; the other comes from intersecting an extension-field covering with the standard copy of $\mathbb F_q^n$.
\begin{thm}\label{th4.3}
For all prime powers $q$ and positive integers $m,n$,
\begin{equation}\label{eq:cross-field-bounds}
 m\,h_{q^m}(n)\leq h_{\tau_3}(\mathbb F_q^{mn})\leq h_q(mn),
\end{equation}
and
\begin{equation}\label{eq:extension-monotonicity}
 h_{q^m}(n)\leq h_q(n).
\end{equation}
Consequently,
\[
 h_{q^m}(n)\leq
 \min\left\{h_q(n),\left\lfloor\frac{h_q(mn)}m\right\rfloor\right\}.
\]
In particular, $h_q(mn)\leq m-1$ implies $h_{q^m}(n)=0$. If $h_q(mn)=0$, then also $h_{\tau_3}(\mathbb F_q^{mn})=0$.
\end{thm}

\begin{proof}
	The first inequality in~\eqref{eq:cross-field-bounds} is
	Theorem~\ref{th4.1}. For the second, recall that $\tau_3=\tau_2^m$.
	Thus the $\tau_3$-images of any subspace are among its $\tau_2$-images.
	Consequently, every $\tau_3$-covering subspace is also
	$\tau_2$-covering, with the same codimension. Taking maximum
	codimensions proves the second inequality.
	
	We next prove~\eqref{eq:extension-monotonicity}. Let
	$U\leq\mathbb F_{q^m}^n$ be a cyclically covering subspace with
	\[
	\operatorname{codim}_{\mathbb F_{q^m}}U=r,
	\]
	and set $U_0=U\cap\mathbb F_q^n$.
	Since the cyclic shift $\tau$ preserves $\mathbb F_q^n$, intersecting
	the covering union with this subspace gives
	\[
	\mathbb F_q^n
	=\mathbb F_q^n\cap\bigcup_{j=0}^{n-1}\tau^jU
	=\bigcup_{j=0}^{n-1}\tau^j(U\cap\mathbb F_q^n)
	=\bigcup_{j=0}^{n-1}\tau^jU_0.
	\]
	Hence $U_0$ is a cyclically covering subspace over $\mathbb F_q$.
	
	It remains to compare codimensions over the two fields.
	Choose an $\mathbb F_q$-basis $u_1,\ldots,u_s$ of $U_0$.
	The matrix with these vectors as columns has a nonzero
	$s\times s$ minor over $\mathbb F_q$. This minor remains nonzero
	over $\mathbb F_{q^m}$, so the same vectors are linearly independent
	over the extension field. Since they belong to $U$, we obtain
	\[
	\dim_{\mathbb F_q}U_0=s
	\leq\dim_{\mathbb F_{q^m}}U=n-r.
	\]
	Therefore
	\[
	h_q(n)\geq\operatorname{codim}_{\mathbb F_q}U_0\geq r.
	\]
	Maximizing over all such $U$ yields
	\[
	h_{q^m}(n)\leq h_q(n),
	\]
	as required. The remaining assertions follow from these inequalities.
\end{proof}

\begin{coro}\label{coro4.1}
If $n=3$ or $n=2^d$ for a nonnegative integer $d$, then $h_4(n)=0$.
\end{coro}
\begin{proof}
For $n=2^d$, Lemma~\ref{lem1.1}(viii) and~\eqref{eq:extension-monotonicity} give $h_4(n)\leq h_2(n)=0$. For $n=3$, a proper subspace of $\mathbb F_4^3$ has at most $4^2$ elements. Its three cyclic images have at most $3\cdot4^2<4^3$ elements in total and therefore cannot cover the space.
\end{proof}

\begin{coro}\label{coro4.2}
Let $q$ be any prime power and let $m,n$ be positive integers. If $h_q(n)=0$, then
\[
 h_{q^m}(n)=0.
\]
\end{coro}
\begin{proof}
This follows immediately from~\eqref{eq:extension-monotonicity} and the nonnegativity of the covering invariant.
\end{proof}

\section{Conclusions}
Using invariant decompositions and trace pairings, we prove that
$h_q(\ell^t)=0$ whenever $q>2$ is a primitive root modulo the odd
prime power $\ell^t$. We also establish
\[
h_q(2n)=0\quad\Longleftrightarrow\quad h_q(n)=0
\]
for odd $q$ and $n$.

Multiplicative-subgroup constructions give further exact values,
and we determine $h_3(11)=h_3(16)=1$. Finally, the inequality
$h_{q^m}(n)\leq h_q(n)$ shows that zero values are preserved
under field extension.

\section*{Acknowledgments}
The authors thank Tony Newton for pointing out an error in the direct-sum statement in an earlier version.

\end{document}